\documentclass[11pt]{amsart}
 
\usepackage[mathscr]{eucal}
\usepackage{amssymb,amsmath,color}

\theoremstyle{plain}
\newtheorem{thm}{Theorem}

\newtheorem{prop}[thm]{Proposition}

\theoremstyle{definition}
\newtheorem{defn}[thm]{Definition}

\theoremstyle{remark}

\newtheorem*{remark*}{Remark}

        \newcommand{\field}[1]{{\mathbb{#1}}}
        \newcommand{\NN}{\field{N}}
        \newcommand{\ZZ}{\field{Z}}
        
        \newcommand{\RR}{\field{R}}

\allowdisplaybreaks

\begin{document}


\title[Quantization on symplectic manifolds of bounded geometry]{Berezin-Toeplitz quantization on symplectic manifolds of bounded geometry}

\author[Y. A. Kordyukov]{Yuri A. Kordyukov}
\address{Institute of Mathematics\\
         Russian Academy of Sciences\\
         112~Chernyshevsky str.\\ 450008 Ufa\\ Russia} \email{yurikor@matem.anrb.ru}




\begin{abstract}
We establish the theory of Berezin-Toeplitz quantization on symplectic manifolds of bounded geometry. The quantum space of this quantization is the spectral subspace of the renormalized Bochner Laplacian associated with some interval near zero. We show that this quantization has the correct semiclassical limit. 
\end{abstract}

\date{}

 \maketitle

\section{Preliminaries and main results}
The Berezin-Toeplitz quantization is a version of geometric quantization, which was suggested by Berezin in \cite{Berezin74,Berezin75}. Since then, several approaches to Berezin-Toeplitz and geometric quantization of compact symplectic  manifolds have been developed  (see, for instance, survey papers \cite{Ali,Englis,ma:ICMtalk,Schlich10}). For a general compact K\"ahler manifold, the Berezin-Toeplitz quantization was constructed by Bordemann, Meinrenken and Schli\-chen\-maier \cite{BMS94}, using the theory of Toeplitz structures of Boutet de Monvel and Guillemin \cite{BG}. Much less is known about quantization of noncompact symplectic manifolds. 

The main goal of this paper is to establish the theory of  Berezin-Toeplitz quantization on a large class of noncompact symplectic manifolds. More precisely, we consider a symplectic manifold $(X,\omega)$ of dimension $2n$, satisfying the following assumptions: 
\begin{itemize}
\item[(i)] $X$ is endowed with a Riemannian metric $g$ such that $(X,g)$ is a Riemannian manifold of bounded geometry and the symplectic form $\omega$ is uniformly $C^\infty$-bounded on $(X,g)$.
\item[(ii)] $\omega$ is uniformly non-degenerate on $X$. 
\item[(iii)] $(X,\omega)$ is quantizable. 
\end{itemize}

Recall that a Riemannian manifold $(X,g)$ is a manifold of bounded geometry, if the curvature $R^{TX}$ of the Levi-Civita connection $\nabla^{TX}$ of $(X, g)$ and its derivatives of any order are uniformly bounded on $X$ in the norm induced by $g$, and the injectivity radius $r_X$ of $(X, g)$ is positive. In particular, the Riemannian manifold $(X,g)$ is complete.  A differential form on $X$ is uniformly $C^\infty$-bounded, if it and its derivatives of any order are uniformly bounded on $X$ in the norm induced by $g$. 

For $x\in X$, let $B_x : T_xX\to T_xX$ be the skew-adjoint operator such that 
\[
\omega_x(u,v)=g(B_xu,v), \quad u,v\in T_xX. 
\]
The operator $B_x^*B_x : T_xX\to T_xX$ is a positive self-adjoint operator. Then $\omega$ is uniformly non-degenerate on $X$ if
\begin{equation}\label{e:uniform-positive}
\frac{(B_x^*B_xu,u)}{|u|^2} \geq \mu^2_0>0, \quad x\in X, \quad u\in T_xX\setminus \{0\}.
\end{equation} 
Finally, $(X,\omega)$ is quantizable, if there exists a Hermitian line bundle $(L,h^L)$ on $X$ with Hermitian connection $\nabla^L$ such that that the curvature $R^L=(\nabla^L)^2$ of the connection $\nabla^L$ satisfies the prequantization condition:
\begin{equation}\label{e1.1}
iR^L=\omega.
\end{equation}

We consider an auxiliary Hermitian vector bundle $(E, h^E)$ on $X$ with Hermitian connection $\nabla^E : {C}^\infty(X,E)\to {C}^\infty(X,T^*X\otimes E)$. We assume that it has bounded geometry, which means that the curvature $R^E$ of the connection $\nabla^E$ and its derivatives of any order are uniformly bounded on $X$ in the norm induced by $g$ and $h^E$. Observe that, by above assumptions, the bundle $L$ has bounded geometry.


Let $\nabla^{L^p\otimes E}: {C}^\infty(X,L^p\otimes E)\to
{C}^\infty(X, T^*X \otimes L^p\otimes E)$ be the connection
on $L^p\otimes E$ induced by $\nabla^{L}$ and $\nabla^E$.
Denote by $\Delta^{L^p\otimes E}$ the induced
Bochner-Laplacian acting on ${C}^\infty(X,L^p\otimes E)$ by
\begin{equation}\label{e1.4}
\Delta^{L^p\otimes E}=\big(\nabla^{L^p\otimes E}\big)^{\!*}\,
\nabla^{L^p\otimes E},
\end{equation}
where $\big(\nabla^{L^p\otimes E}\big)^{\!*}:
{C}^\infty(X,T^*X\otimes L^p\otimes E)\to
{C}^\infty(X,L^p\otimes E)$ denotes the formal adjoint of 
the operator $\nabla^{L^p\otimes E}$.  
The renormalized Bochner-Laplacian is a differential
operator $\Delta_p$ acting on ${C}^\infty(X,L^p\otimes E)$ by
 \begin{equation}\label{e:Delta_p}
\Delta_p=\Delta^{L^p\otimes E}-p\tau,
  \end{equation}
where $\tau\in {C}^\infty_b(X)$ is given by
 \begin{equation}\label{tau}
 \tau(x)=\frac 12\operatorname{Tr}(B_x^*B_x)^{1/2},\quad x\in X.
 \end{equation}
The renormalized Bochner-Laplacian was introduced by
Guillemin and Uribe in \cite{Gu-Uribe}.
When $(X,\omega,g)$ is a K\"{a}hler manifold, it is twice the
corresponding Kodaira-Laplacian 
$\Box^{L^p\otimes E}=\bar\partial^{L^p\otimes E*}\bar\partial^{L^p\otimes E}$.
 
Since $(X,g)$ is complete, the renormalized Bochner Laplacian $\Delta_p$ is essentially self-adjoint, see \cite[Theorem 2.4]{Kor-ma-ma}. 
We still denote by $\Delta_p$ the unique self-adjoint extension of 
$\Delta_p$ acting on the space ${C}^\infty_c(X,L^p\otimes E)$ of compactly supported smooth sections of $L^p\otimes E$. By \cite[Theorem 1.1]{Kor-ma-ma} (see also \cite[Lemma 1]{ma-ma15}),  there exists $C_L>0$ such that for any $p\in \mathbb N$
the spectrum $\sigma(\Delta_p)$ of $\Delta_p$ in $L^2(X,L^p\otimes E)$ satisfies
 \begin{equation}\label{e1.8}
 \sigma(\Delta_p)\subset \big[\!-C_L,C_L\big]\cup
 \big[2p\mu_0-C_L,+\infty\big).
 \end{equation}

Let $\mathcal H_p\subset L^2(X,L^p\otimes E)$ be the spectral subspace 
of $\Delta_p$ corresponding to $[-C_L,C_L]$ and $P_{\mathcal H_p}: L^2(X,L^p\otimes E)\to \mathcal H_p$ be the corresponding spectral projection of $\Delta_p$ (the generalized Bergman projection).
If $X$ is compact, the spectrum of $\Delta_p$ is discrete
and $\mathcal H_p$ is the subspace spanned by the eigensections of $\Delta_p$ 
corresponding to eigenvalues in $[-C_L,C_L]$. 
The space $\mathcal H_p$ will be the quantum space of our quantization. The idea to deal with such a quantum space was suggested by Borthwick and Uribe in \cite{B-Uribe}.

%
%

Let $F$ be a Hermitian vector bundle on $X$ equipped with a Hermitian connection $\nabla^{F}$. 
The Levi-Civita connection $\nabla^{TX}$ on $(X,g)$ and 
the connection $\nabla^{F}$ define a Hermitian connection
$\nabla^{F} : {C}^\infty(X, (T^*X)^{\otimes j}
\otimes F)\to {C}^\infty(X, (T^*X)^{\otimes (j+1)}
\otimes F)$ on each vector bundle $(T^*X)^{\otimes j} \otimes F$ for 
$j\in \mathbb N$, that allows us to introduce the operator 
$\big(\nabla^{F}\big)^{\!\ell} : 
{C}^\infty(X, F)
\to {C}^\infty(X, (T^*X)^{\otimes \ell} \otimes F)$
for every $\ell\in \mathbb N$.  If $F$ has bounded geometry, we denote by 
${C}^k_b(X, F)$ the space of sections 
$u\in {C}^k(X,F)$ such that
\[
\|u\|_{{C}^k_b}=\sup_{x\in X, \ell\leq k}
\Big|\big(\nabla^{F}\big)^{\!\ell} u(x)\Big| <\infty,
\]
where $|\cdot|_x$ is the norm in
$(T^*_xX)^{\otimes \ell} \otimes F_x$ 
defined by $g$ and $h^F$.

For any $f\in C^\infty_b(X,\operatorname{End}(E))$, the multiplication operator by $f$ yields a bounded operator in $L^2(X,L^p\otimes E)$, which will be also denoted by $f$.  We define the Berezin-Toeplitz quantization of $f$ to be a sequence of linear operators $T_{f,p}: L^2(X,L^p\otimes E)\to L^2(X,L^p\otimes E), p\in \mathbb N,$ given by  
\[
T_{f,p}=P_{\mathcal H_p}fP_{\mathcal H_p}.
\]

\begin{thm}\label{t:comm}
If $f,g\in C^\infty_b(X,\operatorname{End}(E))$, then, for the product of the Toeplitz operators $\{T_{f,p}\}$ and $\{T_{g,p}\}$, we have
\begin{equation}\label{e:TfpTgp-prod}
T_{f,p}T_{g,p}=T_{fg,p}+\mathcal O(p^{-1}), \quad p\to +\infty. 
\end{equation}
Moreover, if $f,g\in C^\infty_b(X)$, then, for the commutator of the operators $\{T_{f,p}\}$ and $\{T_{g,p}\}$, we have
\begin{equation}\label{e:TfpTgp-comm}
[T_{f,p}, T_{g,p}]=i p^{-1}T_{\{f,g\},p}+ \mathcal O(p^{-2}), \quad p\to +\infty, 
\end{equation}
where $\{f,g\}$ is the Poisson bracket on $(X,\omega)$.
\end{thm}

The limit $p\to +\infty$ for Toeplitz operators can be thought of as a semiclassical limit, with semiclassical parameter $\hbar=\frac{1}{p}\to 0$. Thus, we show that this quantization for the symplectic manifold $(X,\omega)$ has the correct semiclassical limit. In the particular case when $(X,\omega)$ is a K\"ahler manifold of bounded geometry with a complex structure $J$ and a Riemannian metric $g$ induced by $\omega$ and $J$, and $L$ and $E$ are holomorphic vector bundles, we get an extension of the  K\"ahler quantization constructed by Bordemann, Meinrenken and Schlichenmaier  \cite{BMS94} to this class of noncompact manifolds. 

In the case when $X$ is a compact symplectic manifold, the Berezin-Toeplitz quantization with the quantum space $\mathcal H_p$ was constructed \cite{B-Uribe} in the almost-K\"ahler case and in \cite{ioos-lu-ma-ma,Kor18} for an arbitrary Riemannian metric. We refer to \cite{charles,HM,ma-ma08a} for different approaches 

For the simplest example of a non-compact symplectic manifold, the complex $n$-dimensional vector space $\mathbb C^n$ with the standard symplectic form and Euclidean metric, the Berezin-Toeplitz quantization  is given by the Toeplitz operators on the Fock space. It was discussed already by Berezin in  \cite{Berezin74,Berezin75} and further analyzed in \cite{Bauer-Coburn16,B98,Coburn92} (see also the references therein). In particular, in \cite{Coburn92}, the Berezin-Toeplitz quantization was constructed for functions from $C^{4n+6}_b(\mathbb C^n)$. The role of uniform continuity and boundednes of derivatives was discussed in \cite{Bauer-Coburn16}. In particular, in \cite{Bauer-Coburn16}, the authors provide an example which indicates that the above theorem in general fails in the case of rapidly oscillating bounded functions. 

The theory of Berezin-Toeplitz quantization was studied for various types of domains in $\mathbb C^n$ (see, for instance, \cite{bauer-coburn-hagger18,bauer-hagger-vas17,BLU93,Englis02,Klimek-Lesniewskii92} and references therein). In \cite{Englis-Upmeier15}, Berezin-Toeplitz quantization was constructed in the general setting of symmetric spaces of compact or noncompact type, both in the real and the complex hermitian case (see also the references in this paper). In \cite[Section 5]{ma-ma08a}, Ma and Marinescu considered a class of complete Hermitian manifolds and constructed a Berezin-Toeplitz quantization of the algebra $C^\infty_{const}(X)$ of smooth functions on $X$ which are constant outside of a compact set. To our knowledge, there are no works on Berezin-Toeplitz quantization for a rather general class of noncompact symplectic manifolds.

In this paper, we follow the approach to Berezin-Toeplitz quantization suggested by Ma and Marinescu \cite{ma-ma:book,ma-ma08} in the case when the quantum space is the kernel of the spin$^c$ Dirac operator (also with an auxiliary bundle). It is based on the Bergman kernel expansion from \cite{dai-liu-ma}. Combined with asymptotic expansions for generalized Bergman kernels from \cite{ma-ma:book,ma-ma08} and weighted estimates, this approach was extended to the case when the quantum space is $\mathcal H_p$ in \cite{Kor18} (see also \cite{ioos-lu-ma-ma} for a slightly different approach). Asymptotic expansions of generalized Bergman kernels on manifolds of bounded geometry are proved in \cite{Kor-ma-ma} (see also \cite{Kor20}). The main contribution of this paper is an adaption of the Toeplitz operator calculus of \cite{ioos-lu-ma-ma,Kor18,ma-ma:book,ma-ma08} to the noncompact setting. For this purpose, we use the techniques of weighted estimates developed in \cite{Kor18,Kor20,Kor-ma-ma}. 

The paper is organized as follows. In Section~\ref{s:Toeplitz}, we describe basic facts of the Toeplitz operator calculus in the current setting, namely, a construction of the algebra of Toeplitz operators and a criterion for Toeplitz operators in terms of their Schwartz kernels. Theorem~\ref{t:comm} follows immediately from these results. Section~\ref{s:proof} is devoted to the proof of the criterion. 

\section{Toeplitz operator calculus}\label{s:Toeplitz}
\subsection{Algebra of Toeplitz operators}
In this section, we describe basic facts of the Toeplitz operator calculus in the current setting.
First, we introduce some weighted $L^2$-spaces. 

As usual, we denote by $L^2(X,L^p\otimes E)$ the space of $L^2$-sections of $L^p\otimes E$ with the $L^2$-norm given by
\begin{equation}\label{e2.5}
\|u\|^2=\int_{X}|u(x)|^2dv_{X}(x), \quad u\in L^2(X,L^p\otimes E),
\end{equation}
where $dv_X$ stands for the Riemannian volume form  of $(X,g)$.

Let $d(x,y)$ be the geodesic distance on $X$. 
We introduce a family $\{d_{y} : y\in X\}$  of Lipschitz functions on $X$ given by
\begin{equation}\label{e:dy}
d_{y}(x) =d(x,y), \quad  x\in X,
\end{equation}
and a family $L^2_{\alpha,y}(X,L^p\otimes E)$, $\alpha\in \mathbb R$, $y\in X$, of weighted $L^2$-spaces as 
\[
L^2_{\alpha,y}(X,L^p\otimes E)=\{u\in C^{-\infty}(X,L^p\otimes E) : e^{\alpha d_y}u\in L^2(X,L^p\otimes E)\}
\]
with the Hilbert norm 
\[
\|u\|_{p,\alpha,y}=\|e^{\alpha d_y}u\|, \quad u\in L^2_{\alpha,y}(X,L^p\otimes E).
\]
We note that, as a topological vector space, $L^2_{\alpha,y}(X,L^p\otimes E)$ is independent of $y$. 

The following definition is an extension of \cite[Definition 4.1]{ma-ma08} 
in the current setting.

\begin{defn}\label{d:Toeplitz}
A Toeplitz operator is a sequence of bounded linear operators $T_p : L^2(X,L^p\otimes E)\to L^2(X,L^p\otimes E)$, $p\in \mathbb N$, satisfying the following conditions:
\begin{description}
\item[(i)] For any $p\in \mathbb N$, we have 
\[
T_p=P_{\mathcal H_p}T_pP_{\mathcal H_p}. 
\]
\item[(ii)] There exists a sequence $g_l\in C^\infty_b(X,\operatorname{End}(E))$ such that 
\[
T_p=P_{\mathcal H_p}\left(\sum_{l=0}^\infty p^{-l}g_l\right)P_{\mathcal H_p}
\]
in the following sense:  for any $K\in \mathbb Z_+$ there exist $\mu>0$ and $C>0$ such that for any $p\in \mathbb N$, $\alpha\in \mathbb R$ with $|\alpha|<\mu\sqrt{p}$ and $y\in X$, 
\[
\left\|T_p-P_{\mathcal H_p}\left(\sum_{l=0}^K p^{-l}g_l\right)P_{\mathcal H_p}: L^2_{\alpha,y}(X,L^p\otimes E) \to L^2_{\alpha,y}(X,L^p\otimes E) \right\|\leq Cp^{-K-1}.
\]
\end{description}
  \end{defn}
  
The following theorem states an asymptotic expansion of the composition of Toeplitz operators $T_{f,p}$ and $T_{g,p}$. It is an extension of \cite[Theorem 1.1]{ma-ma08} 
in the current setting.

\begin{thm}\label{t:algebra}
Let $f,g\in C^\infty_b(X,\operatorname{End}(E))$. Then the composition of the Toeplitz operators $T_{f,p}$ and $T_{g,p}$ is a Toeplitz operator in the sense of Definition~\ref{d:Toeplitz}. More precisely, it admits the asymptotic expansion 
\[
T_{f,p}T_{g,p}=\sum_{r=0}^\infty p^{-r}T_{C_r(f,g),p}, 
\]
with some $C_r(f,g)\in C^\infty_b(X,\operatorname{End}(E))$, where $C_r$ are bidifferential operators. In particular, $C_0(f,g)=fg$ and, for $f,g\in C^\infty_b(X)$,
\[
C_1(f,g)-C_1(g,f)=i\{f,g\}.
\] 
\end{thm}

It is clear that Theorem~\ref{t:comm} is an immediate consequence of this theorem. Theorem~\ref{t:algebra} also shows that the set of Toeplitz operators is an algebra. 

The key ingredient in the proof of Theorem~\ref{t:algebra} is a characterization of Toeplitz operators in terms of asymptotic expansions of their Schwartz kernels, which is stated in the next section. For compact manifolds, this type of characterization was introduced in \cite[Theorem 4.9]{ma-ma08} for Toeplitz operators associated with spin$^c$ Dirac quantization and extended to Toeplitz operators associated with the renormalized Bochner Laplacian in \cite[Theorem 6.5]{Kor18} (see also \cite[Theorem 4.1]{ioos-lu-ma-ma}).

\subsection{Criterion for Toeplitz operators}\label{s:charact}
We will use the normal coordinates near an arbitrary point $x_0\in X$. We denote by $B^{X}(x_0,r)$ and $B^{T_{x_0}X}(0,r)$ the open balls in $X$ and $T_{x_0}X$ with center $x_0$ and radius $r$, respectively. We identify $B^{T_{x_0}X}(0,r_X)$ with $B^{X}(x_0,r_X)$ via the exponential map $\exp^X_{x_0}: T_{x_0}X\to X$. Furthermore, we choose trivializations of the bundles $L$ and $E$ over $B^{X}(x_0,r_X)$,   identifying their fibers $L_Z$ and $E_Z$ at $Z\in B^{T_{x_0}X}(0,r_X)\cong B^{X}(x_0,r_X)$ with the spaces  $L_{x_0}$ and $E_{x_0}$ by parallel transport with respect to the connections $\nabla^L$ and $\nabla^E$ along the curve $\gamma_Z : [0,1]\ni u \to \exp^X_{x_0}(uZ)$.

Let $dv_{T_{x_0}X}$ be the Riemannian volume form of the tangent space $T_{x_0}X$ equipped with the Riemannian metric $g^{T_{x_0}X}$ defined by $g$. Let $\kappa_{x_0}$ be the smooth positive function on $B^{T_{x_0}X}(0,r_X)\cong B^{X}(x_0,r_X)$ defined by the equation
\begin{equation}\label{e:kappa}
dv_{X}(Z)=\kappa_{x_0}(Z)dv_{T_{x_0}X}(Z), \quad Z\in B^{T_{x_0}X}(0,r_X).
\end{equation}

Let $\mathcal P_{x_0}\in C^\infty(T_{x_0}X\times T_{x_0}X)$ be the Bergman kernel in $T_{x_0}X$ (see \cite{ma-ma:book,ma-ma08}). If we choose an orthonormal base $\{e_j : j=1,\ldots,2n\}$ in $T_{x_0}X$ such that  
\[
B_{x_0}e_{2k-1}=a_ke_{2k}, \quad B_{x_0}e_{2k}=-a_ke_{2k-1},\quad k=1,\ldots,n,
\]
then $\mathcal P_{x_0}$ is given by
\begin{multline}
\label{e:Bergman}
\mathcal P_{x_0}(Z,Z^\prime)=\frac{1}{(2\pi)^n}\prod_{j=1}^na_j \exp\left(-\frac 14\sum_{k=1}^na_k(|z_k|^2+|z_k^\prime|^2- 2z_k\bar z_k^\prime) \right), \\ Z,Z^\prime\in \RR^{2n} \cong T_{x_0}X,
\end{multline}
where we use the complex coordinates $z_k=Z_{2k-1}+iZ_{2k}, k=1,\ldots,n$. 

Consider a sequence of linear operators $\Xi_p : L^2(X,L^p\otimes E)\to L^2(X,L^p\otimes E)$ with smooth Schwartz kernels. Denote by $\pi_1$ and $\pi_2$ the projections of $X\times X$ on the first and second factor. The Schwartz kernel of $\Xi_p$ with respect to the Riemannian volume form $dv_{X}$ is a smooth section $\Xi_p(\cdot,\cdot)\in {C}^\infty(X\times X, \pi_1^*(L^p\otimes E)\otimes \pi_2^*(L^p\otimes E)^*)$. Let $TX\times_X TX=\{(Z,Z^\prime)\in T_{x_0}X\times T_{x_0}X : x_0\in X\}$ be the fiberwise product  and  $\pi : TX\times_X TX\to X$ be the natural projection given by $\pi(Z,Z^\prime)=x_0$. The kernel $\Xi_p(x,x^\prime)$ induces a smooth section $\Xi_{p,x_0}(Z,Z^\prime)$ of the vector bundle $\pi^*(\operatorname{End}(E))$ on $TX\times_X TX$ defined for all $x_0\in X$ and $Z,Z^\prime\in T_{x_0}X$ with $|Z|, |Z^\prime|<r_X$.

\begin{defn}[\cite{ma-ma:book,ma-ma08}]\label{d:equiv}
We say that 
\[
p^{-n}\Xi_{p,x_0}(Z,Z^\prime)\cong \sum_{r=0}^k(Q_{r,x_0}\mathcal P_{x_0})(\sqrt{p}Z,\sqrt{p}Z^\prime)p^{-\frac{r}{2}}+\mathcal O(p^{-\frac{k+1}{2}})
\]
with some $Q_{r,x_0}\in \operatorname{End}(E_{x_0})[Z,Z^\prime]$, $0\leq r\leq k$,  depending smoothly and $C^\infty$-boundedly on $x_0\in X$, if there exist $\varepsilon^\prime\in (0,r_X]$, $C_0>0$ and $c_0>0$ with the following property: for any $l\in \mathbb N$, there exist $C>0$ and $M>0$ such that for any $x_0\in X$, $p\geq 1$ and $Z,Z^\prime\in T_{x_0}X$, $|Z|, |Z^\prime|<\varepsilon^\prime$, we have 
\begin{multline}\label{e:equiv}
\Bigg|p^{-n}\Xi_{p,x_0}(Z,Z^\prime)\kappa_{x_0}^{\frac 12}(Z)\kappa_{x_0}^{\frac 12}(Z^\prime) -\sum_{r=0}^k(Q_{r,x_0}\mathcal P_{x_0})(\sqrt{p} Z, \sqrt{p}Z^\prime)p^{-\frac{r}{2}}\Bigg|_{C^{l}_b(X)}\\ 
\leq Cp^{-\frac{k+1}{2}}(1+\sqrt{p}|Z|+\sqrt{p}|Z^\prime|)^Me^{-C_0\sqrt{p}|Z-Z^\prime|}+\mathcal O(e^{-c_0\sqrt{p}}).
\end{multline}
\end{defn}

Here $C^{l}_b(X)$ is the $C^{l}_b$-norm for the parameter $x_0\in X$. 

The following theorem gives a characterization of Toeplitz operators in terms of asymptotic expansions of their Schwartz kernels,

\begin{thm}\label{t:charact}
A family $\{T_p: L^2(X,L^p\otimes E)\to L^2(X,L^p\otimes E)\}$ of bounded linear operators is a Toeplitz operator in the sense of Definition \ref{d:Toeplitz} if and only if it satisfies the following three conditions:
\begin{description}
\item[(i)] For any $p\in \mathbb N$, 
\[
T_p=P_{\mathcal H_p}T_pP_{\mathcal H_p}. 
\]
\item[(ii)] $T_p$ has smooth Schwartz kernel $T_{p}(x, x^\prime)$ with respect to $dv_X$, and there exists $\mu>0$ such that for any $\epsilon_0>0$, we have
\[
\big|T_{p}(x, x^\prime)\big|\leq C e^{-\mu \sqrt{p} \,d(x, x^\prime)},\quad  p\in \mathbb N, \quad x, x^\prime \in X, \quad d(x, x^\prime) >\epsilon_0,
\]
where $C>0$ depends only on $\epsilon_0$.

\item[(iii)] There exists a family of polynomials $\mathcal Q_{r,x_0}\in \operatorname{End}(E_{x_0})[Z,Z^\prime]$, depending smoothly and $C^\infty$-boundedly on $x_0$, of the same parity as $r\in \ZZ_+$ and $\varepsilon^\prime\in (0,r_X/4)$ such that, for any $k\in \mathbb N$, $x_0\in X$, $Z,Z^\prime\in T_{x_0}X$, $|Z|, |Z^\prime|<\varepsilon^\prime$,  
\[
p^{-n}T_{p,x_0}(Z,Z^\prime)\cong \sum_{r=0}^k(\mathcal Q_{r,x_0}\mathcal P_{x_0})(\sqrt{p}Z,\sqrt{p}Z^\prime)p^{-\frac{r}{2}}+\mathcal O(p^{-\frac{k+1}{2}}).
\] 
\end{description}
\end{thm} 

The proof of Theorem~\ref{t:charact} will be given in the next section. 
Once we have proved Theorem~\ref{t:charact}, the proof of Theorem \ref{t:algebra} can be easily completed, following the arguments of \cite[Section 4.3]{ma-ma08a}, so we will omit it. 


\section{Proof of criterion for Toeplitz operators}\label{s:proof}

This section is devoted to the proof of the criterion for Toeplitz operators, Theorem~\ref{t:charact}.  In Section \ref{s:2}, we prove the necessary conditions for $\{T_p\}$ to be a Toeplitz operator (the "only if" part), and in Section \ref{s:3}, the sufficient ones (the "if" part).

\subsection{Kernels of Toeplitz operators}\label{s:2}
In this section, we assume that $\{T_p\}$ is a Toeplitz operator in the sense of Definition \ref{d:Toeplitz} and prove that it satisfies the conditions (i)--(iii) of Theorem~\ref{t:charact}. First, we need to introduce appropriate Sobolev spaces (see \cite{Kor-ma-ma} for more detail). 

For any integer $m\geq 0$, we introduce the norm $\|\cdot\|_{p,m}$ 
on ${C}^\infty_c(X,L^p\otimes E)$ by the formula
\begin{equation}\label{e2.6}
\|u\|^2_{p,m}=\sum_{\ell=0}^m \int_{X} \left|\Big(\frac{1}{\sqrt{p}}
\nabla^{L^p\otimes E}\Big)^\ell u(x)\right|^2 dv_{X}(x), 
\quad u\in H^m(X,L^p\otimes E).
\end{equation}
The completion of ${C}^\infty_c(X,L^p\otimes E)$
with respect to $\|\cdot\|_{p,m}$ is the Sobolev space 
$H^m(X,L^p\otimes E)$ of order $m$. For any integer $m<0$, 
we define the norm in the Sobolev space $H^m(X,L^p\otimes E)$ by duality.

To construct weighted Sobolev space, we use the smoothed distance function constructed in  \cite[Proposition 4.1]{Kor91} (see also \cite[Section 3.1]{Kor-ma-ma}). It is a function $\widetilde{d}_p\in C^\infty(X\times X)$, $p\in \mathbb N$, satisfying the following conditions:

(1) we have
\begin{equation}\label{(1.1)}
\vert \widetilde{d}_p(x,y) - d (x,y)\vert  < \frac{1}{\sqrt{p}}\;,
\quad x, y\in X;
\end{equation}

(2) for any $k>0$, there exists $c_k>0$ such that
\begin{equation}
\label{dist}
\left(\frac{1}{\sqrt{p}}\right)^{k-1} \left| \nabla^k_{x}
\widetilde{d}_p(x,y)\right| < c_{k}\:,\quad x, y\in X.
\end{equation}
 We get a family $\{\widetilde{d}_{p,y} : y\in X\}$  of weight functions on $X$ given by
\begin{equation}\label{e:3.10}
\widetilde{d}_{p,y}(x) = \widetilde{d}_p(x,y), \quad  x\in X.
\end{equation}

We introduce weighted Sobolev space $H^m_{\alpha,y}(X,L^p\otimes E)$ as 
\[
H^m_{\alpha,y}(X,L^p\otimes E)=\{u\in C^{-\infty}(X,L^p\otimes E) : e^{\alpha \widetilde{d}_{p,y}}u\in H^m(X,L^p\otimes E)\}
\]
with the Hilbert norm 
\[
\|u\|_{p,m,\alpha,y}=\|e^{\alpha \widetilde{d}_{p,y}}u\|_{p,m}.
\]
By \eqref{(1.1)}, $H^0_{\alpha,y}(X,L^p\otimes E)=L^2_{\alpha,y}(X,L^p\otimes E)$ for any $y\in X$ and $\alpha\in \mathbb R$, and the norm $\|u\|_{p,0,\alpha,y}$ is equivalent to the $L^2$-norm $\|u\|_{p,\alpha,y}$ uniformly in $y\in X$ and $\alpha\in \mathbb R$ with $|\alpha|<c\sqrt{p}$ for any $c$. 

It follows from \cite{Kor-ma-ma} (see also \cite{Kor20}) that there exist $\mu_P>0$ and $p_0\in \NN$ such that, for any $p>p_0$, $\alpha\in \RR$ with $|\alpha|<\mu_P\sqrt{p}$, $m_1. m_2\in \mathbb N$ and $y\in X$, the operator $P_{\mathcal H_p}$ maps $H^{m_1}_{\alpha,y}(X, L^p\otimes E)$ to $H^{m_2}_{\alpha,y}(X, L^p\otimes E)$ with the following norm estimate
\begin{equation}\label{e:PHp}
\left\|P_{\mathcal H_p} : H^{m_1}_{\alpha,y}(X, L^p\otimes E)\to H^{m_2}_{\alpha,y}(X, L^p\otimes E)\right\|\leq C,
\end{equation}
where $C>0$ is independent of $p$, $\alpha$ and $y$.

\begin{prop}\label{t:STp}
Suppose that $K_p$, $p\in \mathbb N$, is  a sequence of bounded linear operators in $L^2(X,L^p\otimes E)$ such that, for any $p\in \mathbb N$, 
\begin{equation}\label{e:PKP}
K_p=P_{\mathcal H_p}K_pP_{\mathcal H_p}
\end{equation}
and there exist $\mu>0$ and $C>0$ such that for any $p\in \mathbb N$, $\alpha\in \mathbb R, |\alpha|<\mu\sqrt{p}$ and $y\in X$, $K_p$ defines a bounded operator in $L^2_{\alpha,y}(X,L^p\otimes E)$ with the norm estimate
\begin{equation}\label{e:Kp-est}
\left\|K_p: L^2_{\alpha,y}(X,L^p\otimes E) \to L^2_{\alpha,y}(X,L^p\otimes E) \right\|\leq C, 
\end{equation}
where $C>0$ is independent of $p$, $\alpha$ and $y$.  Then $K_p$ has smooth Schwartz kernel $K_{p}(x, x^\prime)$, and there exists $\mu_0>0$  such that for any $k\in \mathbb N$, we have
\begin{equation}\label{e:Tfp1}
\big|K_{p}(x, x^\prime)\big|_{{C}^k}\leq C_k p^{n+\frac{k}{2}}
e^{-\mu_0 \sqrt{p} \,d(x, x^\prime)},
\end{equation}
for any $p>p_0$ and $x, x^\prime \in X$ with the constant $C_k>0$, independent of $p$, $x$ and $x^\prime$.
\end{prop}

Here $|K_{p}(x, x^\prime)|_{{C}^k}$ denotes the pointwise ${C}^k$-seminorm of the section $K_{p}\in {C}^\infty(X\times X, \pi_1^*(L^p\otimes E)\otimes \pi_2^*(L^p\otimes E)^*)$ at the point $(x, x^\prime)\in X\times X$, defined by $h^L, h^E$ and $g$ and the connections $\nabla^{L^p\otimes E}$ ans $\nabla^{TX}$.

\begin{proof}
Using \eqref{e:PHp}, \eqref{e:PKP} and \eqref{e:Kp-est}, one can easily show that, for any $p\in \mathbb N$, $p>p_0$,  $|\alpha|<\mu_0\sqrt{p}$ with $\mu_1:=\min(\mu,\mu_P)$, and $m_1. m_2\in \mathbb N$, 
the operator $K_{p}$ maps $H^{m_1}_{\alpha,y}(X, L^p\otimes E)$ to $H^{m_2}_{\alpha,y}(X, L^p\otimes E)$ with the following norm estimates:
\begin{equation}\label{e:mm+2}
\left\|K_{p} : H^{m_1}_{\alpha,y}(X, L^p\otimes E)\to H^{m_2}_{\alpha,y}(X, L^p\otimes E)\right\|\leq C_m,
\end{equation}
where $C_m>0$ is independent of $p$, $\alpha$ and $y$.

Now we can use a refined form of the Sobolev embedding theorem as in the proof of \cite[Theorem 3.6]{Kor-ma-ma} and derive the estimates \eqref{e:Tfp1} with any $\mu_0\in (0,\mu_1)$. 
\end{proof}

By Proposition~\ref{t:STp}, $\{T_p\}$ satisfies the condition (ii) of Theorem~\ref{t:charact} with some $\mu_0>0$. We also see, for any $K\in \mathbb N$, that the Schwartz kernel of the remainder 
\[
R_{K,p}=T_p-P_{\mathcal H_p}\left(\sum_{l=0}^K p^{-l}g_l\right)P_{\mathcal H_p}, \quad p\in \NN,
\]
in the asymptotic series satisfies the estimate
\[
\big|R_{K,p}(x, x^\prime)\big|_{{C}_b^k}\leq C_k p^{n-K-1+\frac{k}{2}}
e^{-\mu_0\sqrt{p} \,d(x, x^\prime)}, \quad p\in \mathbb N,\quad x, x^\prime \in X, 
\]
for any $k\geq 0$. In particular, for the local Schwartz kernel $R_{K,p,x_0}(Z,Z^\prime)$ we get for any $x_0\in X$, $Z,Z^\prime\in T_{x_0}X$, $|Z|, |Z^\prime|<\varepsilon^\prime$,  
\[
\left|p^{-n}R_{K,p,x_0}(Z,Z^\prime)\right|_{C^{k}_b(X)} \leq Cp^{-M}e^{-\mu_0\sqrt{p}|Z-Z^\prime|},
\]
with $M=K+1-\frac{k}{2}$, which can be made arbitrarily small with an appropriate choice of $K$, depending on $k$. This allows us to reduce our considerations to an operator of the form $T_{f,p}$, $f\in C^\infty_b(X,\operatorname{End}(E))$. 

The Schwartz kernel of $T_{f,p}$ is given by 
\[
T_{f,p}(x,x^\prime)=\int_X P_p(x,x^{\prime\prime})f(x^{\prime\prime})P_p(x^{\prime\prime},x^{\prime})dv_X(x^{\prime\prime}),
\] 
where $P_{p}(x,x^\prime)$ is the Schwartz kernel of the operator $P_{\mathcal H_p}$ (the generalized Bergman kernel).

By \cite[Theorem 1.1]{Kor-ma-ma} (see also Proposition~\ref{t:STp}), $P_{p}(x,x^\prime)$ satisfies the condition (ii) of Theorem~\ref{t:charact}. By \cite[Theorem 4.3]{Kor-ma-ma}, for any $k\in \mathbb N$, we have the following asymptotic expansion:
\[
p^{-n}P_{p,x_0}(Z,Z^\prime)\cong
\sum_{r=0}^k(F_{0,r,x_0}\mathcal P_{x_0})(\sqrt{p} Z, \sqrt{p}Z^\prime)p^{-\frac{r}{2}}+\mathcal O(p^{-\frac{k+1}{2}}).
\]
Using these facts, we can proceed as in \cite[Section 4.1]{ma-ma08a} to show that the Schwartz kernel of $T_{f,p}$ satisfies the condition (iii) of Theorem~\ref{t:charact}. This completes the proof.

\subsection{Characterization of Toeplitz operators}\label{s:3}

In this section, we assume that a sequence $\{T_p\}$ satisfies the conditions (i)--(iii) of Theorem~\ref{t:charact} and prove that it is a Toeplitz operator in the sense of Definition \ref{d:Toeplitz}. This is an extension of \cite[Theorem 4.9]{ma-ma08a}, and we will follow the lines of its proof. 

Without loss of generality, we can assume that $T_p$ is self-adjoint
We start with analysis of the full off-diagonal asymptotic expansion for the Schwartz kernel of $T_p$ given by the condition (iii). It is local and, therefore, actually the same as in the compact case. By \cite[Proposition 4.11]{ma-ma08a} (see also \cite{Kor20} for a different proof), for the leading coefficient $Q_{0,x_0}(Z,Z^\prime)$ in this expansion, we get
\[
Q_{0,x_0}(Z,Z^\prime)=Q_{0,x_0}(Z,Z^\prime)=:Q_{x_0}
\]
for any $x_0\in X$ and $Z,Z^\prime\in T_{x_0}X$ with some $Q_{x_0}\in \operatorname{End}(E_{x_0})$. 
We define a section $g_0\in C^\infty(X,\operatorname{End}(E))$, setting 
\[
g_0(x_0)= \mathcal Q_{x_0}, \quad x_0\in X.
\] 
Since the family $\mathcal Q_{0,x_0}$  is $C^\infty$-bounded in $x_0$, it is easy to see that $g_0\in C^\infty_b(X,\operatorname{End}(E))$.

Next, by \cite[Proposition 4.17]{ma-ma08a}, we have  for any $x_0\in X$, $Z,Z^\prime\in T_{x_0}X$, $|Z|, |Z^\prime|<\varepsilon^\prime$,  
\[
p^{-n}(T_p-T_{g_0,p})_{x_0}(Z, Z^\prime)\cong \mathcal O(p^{-1}). 
\]

Now we have to convert this pointwise estimate of the Schwartz kernel into operator estimates in weighted $L^2$-spaces.  This is based on the following proposition. 

\begin{prop}\label{p:Kp-point}
Let $K_p : C^\infty_{c}(X,L^p\otimes E) \to C^\infty(X,L^p\otimes E)$, $p\in \mathbb N$, be a sequence of linear operator with smooth Schwartz kernel $K_{p}(x, x^\prime)$. Suppose that there exists $\mu >0$ such that for any $\epsilon_0>0$ 
\begin{equation}\label{e:Kp-point1}
|K_{p}(x, x^\prime)|\leq C e^{-\mu\sqrt{p} \,d(x, x^\prime)}, \quad p\in \mathbb N,\quad x, x^\prime \in X,\quad d(x, x^\prime) >\epsilon_0,
\end{equation}
where $C>0$ is independent of $p$, $x, x^\prime$.
Moreover, suppose that, for any $x_0\in X$, $Z,Z^\prime\in T_{x_0}X$, $|Z|, |Z^\prime|<\varepsilon^\prime$,  we have
\begin{equation}\label{e:Kp-point2}
p^{-n}K_{p,x_0}(Z, Z^\prime)\cong \mathcal O(p^{-1}). 
\end{equation}
Then there exist $\mu_1>0$ and $p_0\in \mathbb N$ such that for any $p>p_0$, $\alpha\in \RR$ with  $|\alpha|<\mu_1 \sqrt{p}$ and $y\in X$, the operator $K_p$ defines a bounded linear operator in $L^2_{\alpha,y}(X,L^p\otimes E)$ with the norm estimate 
\[
\left\|K_p : L^2_{\alpha,y}(X,L^p\otimes E) \to L^2_{\alpha,y}(X,L^p\otimes E) \right\|< C,
\] 
where $C>0$ is independent of $p$, $\alpha$ and $y$.
\end{prop}

\begin{proof}
For any  $\alpha\in \mathbb R$ and $y\in X$, we have a unitary isomorphism 
$L^2_{\alpha,y}(X,L^p\otimes E)\to L^2(X,L^p\otimes E)$ given by multiplication by $e^{\alpha d_y}$.
Therefore, it is equivalent to show that, for any $p\in \mathbb N$, $\alpha\in \RR$ with $|\alpha|<\mu \sqrt{p}$ and $y\in X$, the operator $e^{\alpha d_y} K_p e^{-\alpha d_y}$ defines a bounded linear operator in $L^2(X,L^p\otimes E)$ and 
\[
\left\|e^{\alpha d_y} K_p e^{-\alpha d_y} : L^2(X,L^p\otimes E) \to L^2(X,L^p\otimes E) \right\|\leq C, 
\]
where $C>0$ is independent of $p$, $\alpha$ and $y$.

The Schwartz kernel $K_{p,\alpha,y}(x, x^\prime)$ of the operator $e^{\alpha d_y} K_p e^{-\alpha d_y}$ is given by
\[
K_{p,\alpha,y}(x, x^\prime) =e^{\alpha d(x,y)} K_p(x, x^\prime) e^{-\alpha d(x^\prime,y)}, \quad x, x^\prime \in X.
\]
By Schur's test, we have
\begin{multline*}
\left\|e^{\alpha d_y} K_p e^{-\alpha d_y} : L^2(X,L^p\otimes E) \to L^2(X,L^p\otimes E) \right\|\\
\leq \max\left(\sup_{x\in X}\int_X |K_{p,\alpha,y}(x, x^\prime)| dv_X(x^\prime), \sup_{x^\prime\in X}\int_X |K_{p,\alpha,y}(x, x^\prime)| dv_X(x) \right).
\end{multline*}

Let us estimate the first integral in the right hand side of the last formula. For the second integral, we have similar estimates.  
By \eqref{e:Kp-point1}, it follows that for any $p\in \mathbb N$, $x, x^\prime \in X$ with $d(x, x^\prime) >\epsilon_0$, $y\in X$ and $\alpha\in \mathbb R$, we have 
\begin{align*}
|K_{p,\alpha,y}(x, x^\prime)| \leq & C e^{\alpha d(x,y)} e^{-\mu \sqrt{p} \,d(x, x^\prime)} e^{-\alpha d(x^\prime,y)}\\ = & C e^{\alpha (d(x,y)-d(x, x^\prime)-d(x^\prime,y))} e^{(\alpha-\mu\sqrt{p}) \,d(x, x^\prime)} \leq  C e^{(\alpha-\mu\sqrt{p}) \,d(x, x^\prime)},
\end{align*}
if $\alpha>0$ and
\[
|K_{p,\alpha,y}(x, x^\prime)| \leq C e^{\alpha (d(x,y)+d(x, x^\prime)-d(x^\prime,y))} e^{(-\alpha-\mu\sqrt{p}) \,d(x, x^\prime)}\leq C e^{(-\alpha-\mu\sqrt{p}) \,d(x, x^\prime)}
\]
if $\alpha<0$. Therefore, for any $\mu_1\in (0,\mu)$ there exists $p_0\in \NN$ such that for any $p>p_0$, $y\in X$ and $\alpha\in \mathbb R$ with $|\alpha|<\mu_1 \sqrt{p}$ 
\begin{multline*}
\sup_{x\in X} \int_{d(x, x^\prime) >\epsilon_0} |K_{p,\alpha,y}(x, x^\prime)| dv_X(x^\prime)\\ <C \sup_{x\in X} \int_{d(x, x^\prime) >\epsilon_0} e^{-(\mu-\mu_1)\sqrt{p} \,d(x, x^\prime)} dv_X(x^\prime)<Ce^{-c_0\sqrt{p}},
\end{multline*}
where $C>0$ and $c_0>0$ are independent of $p$, $y$ and $\alpha$. 

On the other hand, using \eqref{e:Kp-point2}, \eqref{e:equiv} and bounded geometry conditions, it is easy to check that there exists $\mu_2>0$ such that for any $p>p_0$, $y\in X$ and $\alpha\in \RR$ with $|\alpha|<\mu_2 \sqrt{p}$,
\begin{multline*}
\sup_{x\in X} \int_{d(x, x^\prime) <\epsilon^\prime} |K_{p,\alpha,y}(x, x^\prime)| dv_X(x^\prime)\\ \leq C_1 \sup_{x_0\in X} \int_{|Z^\prime|<\varepsilon^\prime} |K_{p,\alpha,y,x_0}(0, Z^\prime)| dv_{T_{x_0}X}(Z^\prime) =\mathcal O(p^{-1}).
\end{multline*}
This completes the proof.
\end{proof}

By Proposition~\ref{p:Kp-point}, we infer that, for any $\mu_1\in (0,\mu)$, there exists $p_0\in \mathbb N$ such that for any $p>p_0$ and $\alpha, |\alpha|<\mu_1 \sqrt{p}$, the operator $p(T_p-P_{\mathcal H_p}g_0P_{\mathcal H_p})$ defines a bounded linear operator in $L^2_{\alpha}(X,L^p\otimes E)$ and there exists $C>0$ such that, for any $p>p_0$ and $\alpha\in \RR$ with $|\alpha|<\mu_1 \sqrt{p}$, 
\[
\sup_{y\in X} \left\|p(T_p-P_{\mathcal H_p}g_0P_{\mathcal H_p}) : L^2_{\alpha,y}(X,L^p\otimes E) \to L^2_{\alpha,y}(X,L^p\otimes E) \right\|< C.
\] 
Thus, the condition (ii) of Definition~\ref{d:Toeplitz} holds for the operator $T_p$ for $K=0$. 

Consider the operator $p(T_p-P_{\mathcal H_p}g_0P_{\mathcal H_p})$. It satisfies the conditions (i)--(iii) of Theorem~\ref{t:charact} with $\mu_1$. We can apply the above arguments to these operator to conclude that there exists $g_1\in C^\infty(X,\operatorname{End}(E))$ such that for any $\mu_1\in (0,\mu)$ and $p>p_0$ and $\alpha, |\alpha|<\mu_2 \sqrt{p}$, the operator $p^2(T_p-P_{\mathcal H_p}g_0P_{\mathcal H_p})-p P_{\mathcal H_p}g_1P_{\mathcal H_p}$ defines a bounded linear operator in $L^2_{\alpha}(X,L^p\otimes E)$ and there exists $C>0$ such that, for any $p>p_0$ and $\alpha\in \RR$ with $|\alpha|<\mu_2 \sqrt{p}$, 
\[
\sup_{y\in X} \left\|p(T_p-P_{\mathcal H_p}g_0P_{\mathcal H_p}) : L^2_{\alpha,y}(X,L^p\otimes E) \to L^2_{\alpha,y}(X,L^p\otimes E) \right\|< C.
\] 

We can proceed by induction and prove that $T_p$ is a Toeplitz operator in the sense of Definition \ref{d:Toeplitz}.

\end{document}